\documentclass[12pt]{amsart}
\usepackage{amssymb, amsmath}
\usepackage{graphics}
\usepackage{latexsym}
\usepackage{amsmath}
\usepackage{amssymb,amsthm,amsfonts}
\usepackage{amscd}
\usepackage[arrow, matrix, curve]{xy}
\usepackage{syntonly}
\theoremstyle{plain}
\newtheorem{thm}{Theorem}[section]
\newtheorem{theorem}[thm]{Theorem}
\newtheorem{lemma}[thm]{Lemma}
\newtheorem{corollary}[thm]{Corollary}
\newtheorem{proposition}[thm]{Proposition}

\newtheorem*{thm0}{Theorem}
\newtheorem*{thm36}{Theorem 3.6}
\theoremstyle{definition}

\newtheorem{question}[thm]{Question}

\newtheorem{remark}[thm]{Remark}

\newtheorem{definition}[thm]{Definition}

\newtheorem{example}[thm]{Example}

\numberwithin{equation}{thm}

\newcommand{\sH}{{\mathcal H}}

\newcommand{\sL}{{\mathcal L}}

\newcommand{\sO}{{\mathcal O}}

\newcommand{\sT}{{\mathcal T}}
\newcommand{\sU}{{\mathcal U}}
\newcommand{\sV}{{\mathcal V}}
\newcommand{\sW}{{\mathcal W}}

\newcommand{\C}{{\mathbb C}}

\renewcommand{\H}{{\mathbb H}}

\renewcommand{\L}{{\mathbb L}}

\newcommand{\Q}{{\mathbb Q}}
\newcommand{\R}{{\mathbb R}}

\newcommand{\T}{{\mathbb T}}
\newcommand{\U}{{\mathbb U}}
\newcommand{\V}{{\mathbb V}}
\newcommand{\W}{{\mathbb W}}

\newcommand{\Z}{{\mathbb Z}}

\newcommand{\rk}{{\rm rank}}

\newcommand{\SL}{{\rm SL}}

\title[]{a characterization of special subvarieties in orthogonal shimura varieties}
\author{Stefan M{\"u}ller-Stach}
\address{Universit\"{a}t Mainz, Fachbereich 17, Mathematik, 55099 Mainz, Germany}
\email{stach@uni-mainz.de, zuok@uni-mainz.de}
\author{Kang Zuo}

\dedicatory{F\"ur Eckart, in Freundschaft und Dankbarkeit}
\thanks{Supported by SFB/TRR 45 der Deutschen Forschungsgemeinschaft 2007-2011, Fudan University 2007}

\begin{document}

\begin{abstract}
Let $A_g$ be the moduli space of principally polarized abelian varieties of dimension $g$ with some level structure and 
$M^0 \subset A_g$ an orthogonal Shimura variety. We consider a smooth toroidal compactification $M$ of $M^0$ and 
a subvariety $Y \subset M$ intersecting the boundary of $M$ transversely.
Then we give necessary and sufficient conditions of Andr\'e-Oort type 
for $Y$ itself being the compactification of a special subvariety $Y^0 \subset M^0$.
\end{abstract}

\maketitle

\section*{Introduction}

In this paper we want to study a certain weak form of the \emph{Andr\'e-Oort conjecture} extending our previous work with Viehweg \cite{mvz08}.
In order to explain our results we first want to introduce the required notation.

\subsection*{Notation}
Let $A_g:=A_g^{[N]}$ denote a fine moduli scheme of principally
polarized abelian varieties of dimension $g$ with a level
$N$-structure, for some $N \geq 3$. We choose in addition a smooth toroidal 
compactification $\overline{A}_g$ as constructed by Mumford et al. \cite[chap. III]{amrt}, 
such that $S=\partial \overline A_g$ is a divisor with normal crossings. 
We denote by $f: X \to A_g$ the universal family of abelian varieties and by $\V=R^1f_*\Q$ the
local system attached to it. There is a polarized variation of Hodge structures (VHS) defined over $\Q$ 
with underlying local system $\V$ which we also denote by $\V$.
The assumption $N \ge 3$ implies that the monodromies of $\V$ around all components of $S$ are unipotent. 
We consider a smooth projective subvariety $Y \subset \overline A_g$ meeting $S$ transversely and define $Y^0:=Y \cap A_g$. 
Throughout this paper we denote subvarieties contained in the locally symmetric part $A_g$ of $\overline A_g$ with a superscript $0$.


Write $(GSp_{2g},H_g)$ for the pure Shimura datum defining $A_g=A_g^{[N]}$ with level structure given by the compact open subgroup $K(N)$ of $GSp_{2g}(\mathbb{A}_f)$. 
By special subvariety of $A_g$ we mean, as is defined in \cite{ky} and \cite[6.2]{moonen-survey}, a geometrically irreducible component of a Hecke translate of 
the image of some morphism $Sh_K(G,X)\rightarrow A_g=Sh_{K(N)}(GSp_{2g},H_g)$, which is defined by an inclusion of Shimura subdatum $(G,X)\subset(GSp_{2g},H_g)$ 
together with some compact open subgroup $K\subset G(\mathbb{A}_f)$ such that $K\subset K(N)$. More concretely, we abuse the notation $A_g=\Gamma(N)\backslash H_g^+$ 
for a fixed connected component of $Sh_{K(N)}(GSp_{2g},X)$, and we mainly work with subvarieties of $A_g$ that are of the form $\Gamma\backslash X^+$, 
where for some Shimura subdatum $(G,X)$ one has $X^+$ a connected component of $X$, and $\Gamma=G(\mathbb{Q})_+\cap\Gamma(N)$. Note that the center of 
$G(\mathbb{R})$ acts on $X$ trivially, and $X^+$ is homogeneous under $G^{ad}(\mathbb{R})^+$.

Typical cases of special subvarieties are given by moduli subschemes of $A_g$ that classify abelian varieties with PEL data. 
We refer the readers to \cite[Section 4, 5]{kottwitz} and \cite[Section 8, 8.14, 8.15, 8.17, etc.]{milne} for more details. Following the notations in \cite{milne}, 
the subdatum $(G,X)\subset(GSp_{2g},H_g)$ can be given as follows. Consider $B$ a simple $\mathbb{Q}$-algebra endowed with a positive anti-involution $*$, 
and $(V,\psi)$ a symplectic $(B,*)$-module. Let $G$ be the linear $\mathbb{Q}$-group of $B$-linear symplectic similitudes of $V$. 
The following moduli problem of tuple $(A,s,i,\eta_K)$ is representable: 

(i) $A$ is a complex abelian variety, with $\pm s$ a polarization of the Hodge structure $H_1(A,\mathbb{Q})$;

(ii) $i$ a homomorphism $B\rightarrow End(A)\otimes_\mathbb{Q}\mathbb{A}_f$, $\eta:V\otimes_\mathbb{Q}\mathbb{A}_f\simeq H^1(A,\mathbb{Q})\otimes_\mathbb{Q}\mathbb{A}_f$ 
a $B\otimes_\mathbb{Q}\mathbb{A}_f$-linear isomorphism sending $\psi$ to an $\mathbb{A}_f^\times$-multiple of $s$, and $\eta_K$ is a $K$-orbit of $\eta$, 
$K$ being some compact open subgroup of $G(\mathbb{A}_f)$, which is assumed to be sufficiently small so as to preserve a level $N$ structure on $H_1(A,\mathbb{Q})$;

(iii) there exists a $B$-linear isomorphism $a:H_1(A,\mathbb{Q})\rightarrow V$ sending $s$ to a $\mathbb{Q}^\times$-multiple of $\psi$. 

The moduli problem is represented by a Shimura variety $Sh_K(G,X)$, with its canonical map into $A_g$. 

In order to obtain special subvarieties of unitary type, one may take $B$ to be a central simple $E$-algebra, 
with $E$ some CM extension of a totally real number field $F$, such that the restriction of $*$ to $E$ gives the 
complex conjugation fixing $F$. In particular, if one takes $V=B$ as a $\mathbb{Q}$-vector space, with $\psi$ given by some some 
element $q\in B^\times$ such that $\psi(x,y)=tr_{B/\mathbb{Q}}(xqy^*)$ (e.g. $q$ is in $E$ such that $c(q)=-q$). 
Then $G^{der}$ is a $\mathbb{Q}$-form of $Res_{F/Q}SL_m$, with $m=\sqrt{\dim_EB}$, and $G(\mathbb{R})$ is a product of unitary groups, 
whose signatures depend on the signatures of $q$ along different embeddings $F\rightarrow \mathbb{R}$. In order to make $X^+$ an Hermitian 
symmetric space associated to $SU(m-1,1)$, one should choose the data in such a way that $G(\mathbb{R})$ is the product of a unitary group of 
signature $(m-1,1)$ with other unitary groups of signature $(m,0)$. 

In \cite[10.2]{hoermann} a Shimura subdatum $(GSpin(V),X)\subset(GSp(C^+(V),H(C^+(V)))$ of orthogonal type is constructed. 
Note that the special subvarieties obtained from this subdatum are of the form $\Gamma\backslash X^+$, with $X^+$ the Hermitian symmetric 
domain associated to $SO(n-2,2)_\mathbb{R}\simeq(GSpin(V))^{ad}_\mathbb{R}$. This reformulates the results in \cite[Section 4, 5]{D72}, 
which is inspired by \cite{KS67}, where Kuga and Satake constructed a morphism from the moduli variety of K3 surfaces to the Siegel 
moduli variety of abelian varieties (for $n=21$). Deligne's construction in \cite{D72} exactly fits into the formalism of Shimura data 
axiomized later in \cite{deligne-corvallis}, except that it follows the traditional convention of signs for Hodge types.

On the special subvariety defined by $(GSpin(V),X)\subset(GSp(C^+(V)),H(C^+(V)))$ there exists a polarized $\mathbb{Q}$-VHS of type 
$\{(-2,0),(-1,-1),(0,-2)\}$, with Hodge numbers $h^{-2,0}=h^{0,-2}=1,h^{-1,-1}=n-2$. Recall that in \cite[4,5]{D72}, 
from the Shimura datum $(SO(V)\simeq SO(n-2,2),D)$ and the natural representation of $\rho:SO(V)\rightarrow GL(V)$, 
one gets, for any $x\in D$, a polarized $\mathbb{Q}$-HS $(V,\rho\circ x)$ of type $\{(-1,1),(0,0),(1,-1)\}$ with 
Hodge numbers $h^{-1,1}=h^{1,-1}=1,h^{0,0}=n-2$. Note that in $\rho\circ x:\mathbb{S}\rightarrow GL(V_\mathbb{R})$, 
the real multiplicative group $\mathbb{G}_m\subset\mathbb{S}$ acts trivially. Now lift $\rho$ to the natural representation 
$\rho':GSpin(V)\rightarrow GL(V)$. Since the center of $GSpin(V)$ conincides with that of $GL(V)$ and that $GSpin(V)$ is the 
central extension of $SO(V)$ by $\mathbb{G}_m$, we deduce that for any $x\in X$ with respect to the Shimura datum 
$(GSpin(V),X)$, $(V,\rho'\circ x)$ is a polarized $\mathbb{Q}$-HS with types and Hodge numbers prescribed as in the 
beginning of the paragraph, namely shifted from the ones in \cite{D72} by $(-1,-1)$. Consequently, from $\rho'$ one obtains 
a polarized $\mathbb{Q}$-VHS on the special subvariety defined by $(GSpin(V),X)$, with the same Hodge numbers as $(V,\rho'\circ x)$, $\forall x\in X$.

Interested readers may also consider more general cases of indefinite quadratic spaces over a totally real number field, as studied in \cite{kudla}.

Let us explain some notation used in the statement of the following two theorems even if more details can be found in the subsequent sections. 
In this paper, the symbol $S_Z$ always denotes the divisor at infinity for any compactified 
subvariety $Z \subset \overline A_g$, i.e., the intersection $S_Z=Z \cap S$ where $S=\partial \overline A_g$ is the boundary of $\overline A_g$, 
whereas the ''open'' part $Z \setminus S_Z$ is denoted by $Z^0$.
In all considerations and proofs in this paper we will always make the following general assumption: \emph{All divisors $S_Z=Z \cap S$ 
which arise from intersections of images of subvarieties $Z \subset \overline A_g$ with $S=\partial \overline A_g$ are 
divisors with normal crossings, i.e., the intersections are transverse.} In particular we can speak about logarithmic differential forms on $Z$ 
with poles in  $S_Z$. 
Many of our results may hold with weaker assumptions. However, we want to work out the principles here, and do not strive for maximal generality.

Let $Y \subset \overline A_g$ be a smooth subvariety and $W \subset Y$ a subvariety of $Y$ satisfying the above transversality as assumptions. We denote by 
$$
N_{W/Y}=T_Y(-\log S_Y)/T_W(-\log S_W)
$$ 
the logarithmic normal bundle of $W$ in $Y$. Let $\rho$ be the Picard number of $Y$ and $\delta$ 
the number of crossings of $S \cap Y$. Let $i: M^0 \hookrightarrow  A_g$ be a special subvariety for the orthogonal group 
$SO(2,n),$ where $i$ is the so-called Kuga-Satake embedding  \cite{KS67}, and such that its toroidal compactification $M$ also is embedded into $\overline A_g$. 
On $M^0$ there is a natural polarized VHS $\V'$ of weight two and rank $n+2$ coming from the standard representation of $SO(2,n)$ in $GL(n+2)$.  There is a natural 
proper inclusion $\V'\subset i^*\V^{\otimes 2}$ as a polarized sub VHS which is explained for example in \cite{D72}. The local monodromy of $\V'$ 
around $S_M$ is assumed to be unipotent as well. The (canonical) Deligne extension $\overline {\mathcal V}'$ of ${\mathcal V}':=\V' \otimes {\mathcal O}_{M^0}$ to $M$
carries a natural Hodge filtration, i.e., a descending filtration
$$
{\mathcal V}'=F^0 \supset F^1 \supset \cdots 
$$
by subvector bundles and a logarithmic Gauss-Manin connection $\overline{\nabla}:\overline{{\mathcal V}'} \to \overline{{\mathcal V}'} \otimes \Omega^1_M(\log S_M)$ 
extending $\nabla:{\mathcal V}' \to {\mathcal V}' \otimes \Omega^1_M$. 
The graded object associated to this filtration together with the graded logarithmic Gauss-Manin connection $\vartheta$ 
is the corresponding logarithmic Higgs bundle $(E=E^{2,0} \oplus E^{1,1} \oplus E^{0,2},\vartheta)$ 
under the Simpson correspondence \cite[Main Theorem]{simpson}. Note that only for VHS this correspondence is so simple.
Griffiths' transversality for ${\mathcal V}'$ translates into the fact that $\vartheta$ is given by maps

$$
\vartheta^{2,0}:E^{2,0}   \longrightarrow E^{1,1} \otimes \Omega^1_M(\log S_M), \quad 
\vartheta^{1,1}:E^{1,1} \stackrel{\simeq}{\longrightarrow} E^{0,2} \otimes \Omega^1_M(\log S_M)
$$
 
and $\vartheta^{0,2}=0$. Integrability of $\nabla$ implies $\vartheta \wedge \vartheta=0$.

The \emph{Griffiths-Yukawa coupling} $\vartheta^{(2)}_Z$ on a smooth subvariety 
$i:Z\hookrightarrow  M$ intersecting $S_M$ transversely in $S_Z$ is defined as the composition 
$$
\vartheta^{(2)}_Z:=\vartheta^{1,1} \circ \vartheta^{2,0}:i^* E^{2,0} \longrightarrow 
i^*E^{1,1} \otimes \Omega^1_Z(\log S_Z) \longrightarrow i^*E^{0,2} \otimes \Omega^1_Z(\log S_Z)^{\otimes 2}.
$$
Note that $\vartheta^{(2)}$ lands in $S^2 \Omega^1_Z(\log S_Z)$ as the image in 
$i^* E^{0,2} \otimes \Omega^2_Z(\log S_Z)$ is zero by the condition $\vartheta \wedge \vartheta=0$. 

In the following statements, the degree of a vector bundle $F$ with respect to a line bundle $L$ on a smooth projective variety $Y$ of dimension $d$ 
is defined as $\deg_L(F):=c_1(L)^{d-1}c_1(F)$. The slope is defined as $\mu_L(F):=\deg_L(F)/{\rm rank}(F)$.  

\subsection*{Previous and new results}

The \emph{Andr\'e-Oort conjecture} asserts that an irreducible subvariety $Y^0 \subset A_g$ is special
if and only if it contains a dense set of CM points. Klingler and Yafaev \cite{ky} have announced a proof of it using results of Ullmo and Yafaev 
and had to assume a generalized Riemann hypothesis.

Our methods in this paper are not very sensitive to CM points versus non-CM points. However, the Andr\'e-Oort conjecture would also 
imply that the closure of any union of positive dimensional special subvarieties is again special.

Our main goal therefore is to prove the following consequence of the Andr\'e-Oort conjecture: Let 
$Y^0 \subset A_g$ be a subvariety containing \emph{sufficiently many} special subvarieties of dimension $\ge 1$. Then $Y^0$ is special.

The notion ''sufficiently many'' can be expressed for example in the following way. 
In \cite[Thm. 4.4]{mvz08} we used special divisors $W_i^0 \subset Y^0$ satisfying \emph{Hirzebruch-H\"ofer proportionality} (HHP) 
with $i \in I$, a finite index set. Condition (HHP) is an equality condition arising from 
an inequality which in turn has its origin in natural stability conditions for Higgs bundles. We refer to \cite{mvz08}
for the history of this condition. 

In \cite{mvz08} we then showed that $Y^0$ is special if $\sharp I$ exceeds some effective bound: 

\begin{thm0} [Thm. 4.4 in \cite{mvz08}] 
Let $Y \subset M \subset \overline A_g$ be a subvariety of $\overline A_g$ contained in a toroidal compactification $M$ of a 
Shimura subvariety $M^0 \subset A_g$ of type $SO(2,d)$. We assume that $Y$ and $M$ intersect the boundary $S$ of $A_g$ transversely, and 
require that $\Omega^1_Y(\log S_Y)$ is nef and $\omega_Y(S_Y)$ is ample with respect to $Y^0$.  
Assume $\dim(Y) \ge 2$ and $W_i \subset Y$ ($i \in I=$ finite set) are pairwise distinct divisors such that $W_i^0 \subset M^0$ is special.
\begin{enumerate}
\item[(i)] If all $W_i^0$ are of orthogonal type, if all $W_i$ satisfy condition 
$$
(HHP): \quad \frac{\deg_{\omega_{W_i}(S_{W_i})}(N_{W_i/Y})}{\rk N_{W_i/Y}}=
\frac{\deg_{\omega_{W_i}(S_{W_i})}(T_{W_i}(-\log S_{W_i}))}{\rk T_{W_i}(-\log S_{W_i})},
$$
and if $\# I \geq \varsigma(Y):=(\rho+\delta)^2+\rho+\delta+1$, then $Y^0\subset M^0$ is a special subvariety of orthogonal type.
\item[(ii)] Assume that the Griffiths-Yukawa coupling vanishes on $Y$. If the $W_i^0$ are special subvarieties of unitary type, if condition (HHP) holds
$$
(HHP): \quad \frac{\deg_{\omega_{W_i}(S_{W_i})}(N_{W_i/Y})}{\rk N_{W_i/Y}}
= \frac{\deg_{\omega_{W_i}(S_{W_i})}(T_{W_i}(-\log S_{W_i}))}{d+1}, 
$$
and if $\# I \geq \varsigma(Y)$, then $Y^0\subset M^0$ is a special subvariety of unitary type.
\item[(iii)]  Let $Y$ be a surface and $I=\{1,2\}$. Assume that
$$
\sigma_1(W_1) \cap \sigma_2(W_2) \neq \emptyset
$$
and $\deg N_{W_i/Y}=0$ for $i=1,2$. Then $Y^0$ is the product of two Shimura curves.
\end{enumerate} 
\end{thm0}

In the following main result in this paper we remove the divisor hypothesis and obtain necessary and sufficient conditions 
supporting the Andr\'e-Oort conjecture. 

We need some additional notation to explain the theorem: We say that $Y$ can be covered by a smoothing of a cycle 
$\sum_i a_i C_i$ of compactified special curves $C_i^0 \subset Y^0$ satisfying (HHP), if there are finitely many embedded
special curves $C_i^0 \subset Y^0$ satisfying (HHP) such that their compactifications $C_i  \subset Y$ admit a linear combination $\sum_i a_i C_i$
with integer coefficients which can be deformed as embedded cuves in $Y$ in a family, such that the general deformation is smooth. 
Condition (HHP) in this case is given by equality in the following inequality: 

\begin{equation*} \tag{HHP} 
\deg N_{C/Y} \leq \frac{\rk (N^1_{C/Y})+\rk (N^0_{C/Y})}{2} \cdot \deg T_C(-\log S_C).
\end{equation*}

Here $N^\bullet_{C/Y}$ is the Harder-Narasimhan filtration on the logarithmic normal bundle $N_{C/A_g}$ intersected with $N_{C/Y}$.
We also fix $C_1$ and a base point $y_0 \in C_1$. With the notation for the Higgs bundle $E$ on $M$ restricted to $Y$ we then define the following 
vector spaces: $W_{y_0\in Y}$  is the subspace of  vectors in $E^{1,1}_{y_0}$ vanishing unter $\vartheta$ at the base point $y_0$
and $W_{y_0\in Y,\R} \subset W_{y_0\in Y}$ the real subspace of real vectors in $W_{y_0\in Y}$.

\begin{thm36}  Let $Y \subset M \subset \overline A_g$ be a subvariety of $\overline A_g$ contained in a toroidal compactification $M$ of a 
Shimura subvariety $M^0 \subset A_g$ of type $SO(2,n)$. We assume that $Y$ and $M$ intersect the boundary $S$ of $A_g$ transversely, and 
that $Y$ can be covered by a smoothing of a cycle $\sum_i a_i C_i$ of compactified special curves $C_i^0 \subset Y^0$ satisfying (HHP). Then: \\ 
(a) If $W_{y_0\in Y}=W_{y_0\in Y,\R}\otimes \C$ for some $y_0\in C_1$ then $Y^0\subset M^0$ is a special subvariety of orthogonal type.\\
(b) If the Griffiths-Yukawa couplings along all $C_i$ do not vanish then $Y^0\subset M^0$ is a special subvariety of orthogonal type. \\
(c) If the Griffiths-Yukawa coupling along $Y$ vanishes then $Y^0\subset M^0$ is a special subvariety of unitary type, i.e., a ball quotient. 
\end{thm36}

In the assertions (a) and (b) of this theorem one may replace the assumption on the smoothing of the cycle $\sum_i a_i C_i$ of special curves 
$C_i \subset Y$ satisfying (HHP) by the following: Assume that there is a connected union $C_1 \cup \cdots \cup C_l$ of compactified special 
curves satisfying (HHP) and such that the image $\pi_1(\bigcup C_i^0,*)$ under the natural map in $ \pi_1(Y^0,*)$ is \emph{big} in the sense that
the image of $\pi_1(\bigcup C_i^0,*)$ under the monodromy representation  restricted to $Y^0$
$$
\pi_1(\bigcup C_i^0,*) \longrightarrow \pi_1(Y^0,*)\stackrel{\rho}{\rightarrow}SO(2,n)
$$ 
is Zariski dense in the \emph{algebraic monodromy group} $H(Y^0)$, i.e., the $\Q$-algebraic closure of the monodromy representation  $\rho.$ 
We note that this is a subgroup of Hermitian type in ${\rm SO}(2,n)$. This is nicely explained, for example, in \cite[Sect. 1.3]{moonen}.
Hermitian subgroups of ${\rm SO}(2,n)$ like $H(Y^0)$ can be classified. Besides the obvious 
orthogonal and unitary subgroups which are $\Q$-simple there are ${\rm SL}_2 \times {\rm SL}_2$ and quaternionic versions \cite[Thm. 5.2.3.]{sz}.
In the non-$\Q$-simple cases we therefore have $\dim(Y)=2$ and $Y^0$ is uniformized by a product $\H \times \H$ of upper half planes.

\subsection*{Acknowledgements}
This work naturally continues the results in \cite{mvz08}. Together with Eckart Viehweg we have thought about \emph{thickenings
of Higgs bundles} during a stay at Fudan University in the summer of 2007. Thickenings play an essential role in this paper which 
therefore should be considered as joint work with Eckart. \\
We thank Ke Chen for explanations on Shimura data.
We also thank Michael Harris for enlightening discussions about special cycles and the referee for several helpful remarks.

\section{Basic Setup}

In this section we will use the \emph{Simpson correspondence} for curves \cite[Main Thm.]{simpson}. It is a natural equivalence between 
the category of direct sums of stable filtered regular Higgs bundles of degree zero and the category of direct sums of 
stable filtered local systems of degree zero. We will need this correspondence only in the case when the local
system $\V$ has unipotent local monodromies. In that case the filtration on the Higgs bundle is trivial and $\deg(\V)$ is automatically zero.
We refer the reader to \cite[sect. 1]{vz04} for additional results and explanations on Higgs bundles on curves building up on Simpson's work. 

Consider a non-singular projective curve $C$ and a non-constant morphism
$$
\varphi: C \to Y \subset \overline A_g,
$$
where $Y \subset \overline A_g$ is a smooth projective subvariety as in the introduction. We set 
$C^0:=\varphi^{-1}(Y^0)\not=\emptyset$, where $Y^0=Y \cap A_g$ denotes the ''open'' part.

In the following we consider the situation where $C^0$ is a Shimura curve and $S_C:=C \setminus C^0$ is the set of cusps.
We also denote by $S_Y$ the intersection of $Y$ with $S=\partial \overline A_g$ and we assume that the intersection is transversal
such that $S_Y$ is a divisor with normal crossings. We assume that the restriction $\varphi: C^0 \to Y^0 \subset A_g$ is an \'etale 
morphism of Shimura varieties. Let $f: X\to C^0$ denote also the family obtained by pullback via $\varphi.$

The main goal of this paper is to find new criteria when $Y^0$ itself is a special subvariety in $A_g$, for example if 
''sufficently many'' such curves $C$ with certain properties map to $Y$. 
In such a situation, by \cite[Prop. 1.4]{vz04} and \cite[Thm. 0.9]{mvz06}, after replacing $C^0$ by an 
\'etale cover, the local system $\V_{C^0}:=\varphi^*R^1f_*\C_X$ admits a decomposition
$$
\V_{C^0}=\L\otimes \T \oplus \U
$$
as a  polarized  \emph{complex variation} in the sense of Deligne, i.e., a  polarized $\C$-VHS in the sense of Simpson \cite{simpson} on $C^0$. 
Note that this \'etale cover of $C^0$ is necessary, however, all our proofs below 
are insensitive to such \'etale base change even if we apply this construction to a finite number of curves simultaneously later.
Here $\L$ is of weight one and rank two with the  logarithmic Higgs bundle
$$
\left( \sL\oplus \sL^{-1},\tau:\sL\simeq \sL^{-1}\otimes \Omega^1_C(\log S_C) \right),$$

$\T$  is concentrated in bidegree $(0,0)$ and selfdual,
whereas $\U$ is of weight one and decomposes in two local subsystems
$$
\U=\U^{1,0}\oplus \U^{0,1},\quad \U^{1,0}=\U^{0,1\vee}.
$$
Note that the local systems $\T$, $\U^{1,0}$ and $\U^{0,1}$ are unitary and the local monodromies around $S_C$ are unipotent, 
hence the local monodromies are in fact trivial. Hence $\T$, $\U^{1,0}$ and $\U^{0,1}$ can be extended as local systems to $C$.
Writing $(\sT=\T\otimes\sO_C, 0)$ and $(\sU=\U^{1,0}\otimes\sO_C, 0)\oplus (\sU^\vee=\U^{0,1}\otimes\sO_C, 0)$ for the corresponding Higgs
bundles,   then the Higgs bundle corresponding to $\V_{C^0}$ decomposes in the form 
\begin{equation*} \tag{1.1} \label{1.1}
(E^{1,0}\oplus E^{0,1},\theta)=(\sL\oplus \sL^{-1},\tau)\otimes(\sT,0)\oplus (\sU,0)\oplus(\sU^\vee,0). 
\end{equation*}

The line bundle $\sL$ has positive degree, since $\sL$ is the pullback of some positive power of the automorphic line bundle 
on $\overline A_g$ via $\varphi: C\to \overline A_g$. Since $\varphi: C^0\to A_g$ is not constant and the automorphic line bundle is positive on $\overline A_g$ 
it follows that $\deg(\sL)$ is positive. Via the isomorphism $\tau$ we identify $T_C(-\log S_C)=\sL^{-2}$.

The Hodge metric $\V_{C^0}$ comes from the tensor product of the Hodge metrics on $\L,\, \T,\, \U^{1,0}$ and $\U^{0,1}$, 
which, by \cite[Sect. 4]{simpson1} and \cite{simpson}, coincide with the Hermitian-Yang-Mills metrics on the corresponding logarithmic Higgs bundles.  
The Hodge metrics on $E^{1,0}$ and $E^{0,1}$ are tensor products of the Hodge metrics on $\sL^{\pm},\, \sT,\, \sU^{1,0}$ and $\sU^{0,1}$.

In general, a Hodge bundle with Hodge metric of any Schur functor $S(\V_{C^0})$ is obtained in a similar way from the Hodge metrics on 
$\sL^{\pm},\, \sT,\, \sU^{1,0}$ and $\sU^{0,1}$.\\[.1cm]

Let $f: X\to A_g$ denote the universal family, $\V:=R^1f_*\C_X$  and   $E:=E^{1,0} \oplus  E^{0,1}$ 
the logarithmic Higgs bundle corresponding to Deligne's canonical extension  of $\V \otimes \sO_{A_g}$ 
on the toroidal compactification $\overline A_g \supset A_g$. It comes with the logarithmic Higgs map 
$$
\theta: E^{1,0}\to  E^{0,1}\otimes\Omega^1_{\overline A_g}(\log S).
$$ 
Since $\V$ is a polarized VHS, there is a natural isomorphism ${\rm End}(E)\stackrel{\simeq}{\rightarrow}E^{\otimes 2}$. 
Then it is well-known \cite[p. 339]{faltings} that the composition
$$ 
T_{\overline A_g}(-\log S)\stackrel{\theta}{\rightarrow}  {\rm End}(E)\stackrel{\simeq}{\rightarrow}E^{\otimes 2}\to S^2(E^{0,1}),
$$
identifies $T_{\overline A_g}(-\log S)$ with $S^2(E^{0,1})$. 
The derivatives of the maps
$$
C\stackrel{\varphi}{\rightarrow}Y\stackrel{i}{\rightarrow}\overline A_g
$$ 
induce the following commutative diagramm

\begin{equation*} \tag{1.2} \label{1.2}
\xymatrix{0 \rightarrow   T_C(-\log S_C) \ar@{=}[d] \ar[r]^{d \varphi}  &  \varphi^*T_Y(-\log S_Y)  \ar@{^{(}->}[d]^{d i} 
\ar[r]_{}^{} & N_{C/Y} \ar@{^{(}->}[d] \rightarrow 0 \\
 0  \rightarrow   T_C(-\log S_C) \ar[r]^{d (i\circ \varphi)} &(i\circ\varphi)^* T_{\overline A_g}(-\log S_Y) \ar[r]^{}& N_{C/\overline A_g} \rightarrow 0 },
\end{equation*}

where $N_{C/Y}$ is the (logarithmic) normal bundle of $\varphi: C\to Y$ and $N_{C/\overline A_g}$ is the (logarithmic) normal bundle
of $i\circ\varphi: C\to \overline A_g.$

On the curve $C$ one has 
$$ (i\circ\varphi)^* T_{\overline A_g}(-\log S)=(i\circ\varphi)^*S^2(E^{0,1})
=\left( \sL^{-2}\otimes S^2(\sT) \right) \oplus \left(\sL^{-1}\otimes \sT\otimes\sU^\vee\right) \oplus S^2(\sU^\vee),
$$
where the decomposition on the right side is induced by \eqref{1.1} and is orthogonal with respect to the Hodge metric.\\

All three summands are polystable by the main theorem in \cite{simpson}, but, as $\deg(\sL)>0$,  with different slopes
$$
-2\deg \sL,\; -\deg \sL \textrm{ and } 0.  $$

Consider the inclusion 
$$
T_C(-\log S_C)\stackrel{d\varphi}{\rightarrow}\varphi^*T_Y(-\log S_Y)\stackrel{di}{\rightarrow}(i\circ\varphi)^*T_{\overline A_g}(-\log S)=
$$
$$
= \left( \sL^{-2}\otimes S^2(\sT) \right) \oplus \left(\sL^{-1}\otimes \sT\otimes\sU^\vee\right) \oplus S^2(\sU^\vee),            
$$
As the derivative  $d\varphi$ can be identified with the Higgs map $\theta$ and  $\theta$ on $C$ preserves the direct sum decomposition in \eqref{1.1}
and vanishes on the second summand,  the image of $T_C(-\log S_C)$ is contained in $\sL^{-2}\otimes S^2(\sT)$.\\[.1cm]
For the convenience of the reader we recall the following definition. 

\begin{definition} A holomorphic subbundle $i:F\hookrightarrow E$ of a Hodge bundle $E$ of a polarized complex variation of Hodge structure 
is called a direct summand of $E$ and orthogonal with respect to the Hodge metric if there 
exists an isomorphism $E\simeq F\oplus G$ between holomorphic vector bundles, such that the first summand defines the inclusion $i$  and 
the decomposition is orthogonal with respect to the Hodge metric. 
\end{definition}

\begin{lemma} \label{1.0}  The line subbundle
$$ T_C(-\log S_C)\subset\sL^{-2}\otimes S^2(\sT)$$
induces a holomorphic decomposition of $\sL^{-2}\otimes S^2(\sT)$, which is  orthogonal with respect the Hodge metric, i.e.,
there exists a holomorphic subbundle
$$
T_C(-\log S_C)^\bot\subset \sL^{-2}\otimes S^2(\sT)
$$ 
such that
$$ \sL^{-2}\otimes S^2(\sT)=T_C(-\log S_C)\oplus T_C(-\log S_C)^\bot,$$
and such that this decomposition is orthogonal with respect to the Hodge metric.
\end{lemma}

\begin{proof} We note first that the Hodge metric on $\sL^{-2}\otimes S^2(\sT)$ comes from  the corresponding tensor product of the  Hodge metrics
on the polarized $\mathbb C-$VHS $\mathbb L$, $\mathbb T$ and $\mathbb U$. The Hodge metric on the corresponding Higgs bundle $\sT$ is the Hermitian-Yang-Mills metric
by \cite{simpson}. Consider the subbundle
$$
T_C(-\log S_C)\subset\sL^{-2}\otimes S^2(\sT).
$$
Since $$\tau: \sL\simeq \sL^{-1}\otimes\Omega^1_C(\log S_C),$$
we have $T_C(-\log S_C)=\sL^{-2}$,  hence 
$$
\sL^{-\otimes 2}\subset \sL^{-2}\otimes S^2(\sT).
$$
Dividing both sides by the factor $\sL^{-\otimes 2}$ we get
$$
\sO_C\subset S^2(\sT).
$$
Note that the Higgs bundle $S^2(\sT)$ has zero Higgs field.  Hence $\sO_C$ is a Higgs sub bundle of $S^2(\sT)$ with slope equality $\mu(\sO_C)=0=\mu(S^2(\sT))$.
Applying Simpson's Higgs polystability, there exists a holomorphic decomposition 
$$  
S^2(\sT)=\sO_C\oplus \sO_C^\bot,
$$
which is orthogonal w.r.t. the Hermitian-Yang-Mills metric on $S^2(\sT).$
Tensoring with $\sL^{-\otimes 2}$ on both sides of the above decomposition, we obtain the desired decomposition as claimed. 
\end{proof}

The decomposition in Lemma~\ref{1.0} induces the following decompostion
$$
(i\circ\varphi)^* T_{\overline A_g}(-\log S_Y)=T_C(-\log S_C)\oplus \left( T_C(-\log S_C)^\bot\oplus (\sL^{-1}\otimes \sT\otimes\sU^\vee) \oplus S^2(\sU^\vee)\right).
$$
Let $p$ denote the projection to the first summand, then  the composition
$$ 
T_C(-\log S_C) \stackrel{ d  \varphi  }{\rightarrow}\varphi^*T_Y(-\log S_Y)\stackrel{di}{\rightarrow}
(i\circ \varphi)^* T_{\overline A_g}(-\log S)   \stackrel{p}{\rightarrow}T_C(-\log S_C)
$$
is the identity. This shows that both horizontal short exact sequences in diagram (\ref{1.2})  split in the form

\begin{equation*} \tag{1.3} \label{1.3}
\xymatrix{ \varphi^*T_Y(-\log S_Y)  \ar@{^{(}->}[d]^{d i}  \ar@{=}[r] & T_C(-\log S_C) \ar@{=}[d]    
\oplus &{ \ar@{^{(}->} (62.5,-17.5)*+{N_{C/\overline A_g}}; (62.5,0)*+{  N_{C/Y}}; } \\
 (i\circ\varphi)^* T_{\overline A_g}(-\log S_Y)\ar@{=}[r]  & T_C(-\log S_C)    \oplus &  }
\end{equation*}

such that   
$$
N_{C/\overline A_g}=  T_C(-\log S_C)^\bot\oplus (\sL^{-1}\otimes \sT\otimes\sU^\vee) \oplus S^2(\sU^\vee).
$$

\begin{remark} The holomorphic and orthogonal splitting 
$$
T_C(-\log S)\stackrel{d(i\circ\varphi)}{\rightarrow} (i\circ \varphi)^*T_{\overline A_g}(-\log S)
$$
for a special curve $i\circ \varphi: C^0\to A_g$ in (\ref{1.3}) holds also true in general if $C^0$ is replaced by any special subvariety,
see the proof for ii) in Proposition \ref{prop 1.5}. In diagram (\ref{1.3}) we obtain an explicit description 
of the logarithmic normal bundle $N_{C/\overline A_g}$. 
\end{remark}
 
We shall now describe the Harder-Narasimhan filtration  on $N_{C/\overline A_g}$.
Let 
$$
N_{C/\overline A_g}^0: = T_C(-\log S_C)^\bot,
$$
$$
N_{C/\overline A_g}^1:=   T_C(-\log S_C)^\bot \oplus \sL^{-1}\otimes \sT\otimes\sU^\vee$$
and 
$$
N_{C/\overline A_g}^2:=N_{C/\overline A_g}.
$$
Then  the filtration 
$$
0 \subset  N_{C/\overline A_g}^0\subset N_{C/\overline A_g}^1\subset N_{C/\overline A_g}^2=N_{C/\overline A_g}
$$
is the Harder-Narasimhan filtration  on $N_{C/\overline A_g}$. In our situation, the graded summands 
$$
N_{C/\overline A_g}^{i}/N_{C/\overline A_g}^{i-1}, \text{ for } 0\leq i\leq 2
$$  
are polystable vector bundles of slopes $\deg T_C(-\log S_C)$, $\frac{1}{2} \deg T_C(-\log S_C)$ , and $0$. One has
$$
N_{C/\overline A_g}^1= N^0_{C/\overline A_g}\oplus N^1_{C/\overline A_g}/N^0_{C/\overline A_g},
$$
$$
N^2_{C/\overline A_g}=N^1_{C/\overline A_g}\oplus N^2_{C/\overline A_g}/N^1_{C/\overline A_g}.
$$
Taking the induced filtration on $N_{C/Y} \subset N_{C/\overline A_g}$ obtained by intersection with $N_{C/\overline A_g}^i$
$$ 
0 \subset N_{C/Y}^0 \subset  N_{C/Y}^1\subset N_{C/Y}^2= N_{C/Y},
$$
one finds subbundles
$$ 
N_{C/Y}^{i+1}/N_{C/Y}^i\subset N_{C/\overline A_g}^{i+1}/N_{C/\overline A_g}^i.
$$
We arrive at the following definition: 

\begin{definition}
$\varphi: C\to Y$ satisfies \emph{relative Hirzebruch-H{\"o}fer proportionality} (HHP) if the slope inequalities
$$
\mu(N_{C/Y}^{i+1}/N_{C/Y}^i)\leq\mu (N_{C/\overline A_g}^{i+1}/N_{C/\overline A_g}^i),\quad i=0,1,2
$$
are equalities. One has 
\begin{align*}
\mu (N_{C/\overline A_g}^{2}/N_{C/\overline A_g}^1)    & = \mu(S^2(\sU^\vee))=0, \cr 
\mu (N_{C/\overline A_g}^{1}/N_{C/\overline A_g}^0)    & = \mu(\sL^{-1}\otimes \sT\otimes\sU^\vee)=\frac{1}{2} \deg T_C(-\log S_C), \cr 
\mu (N_{C/\overline A_g}^{0})                          & = \mu(T_C(-\log S_C)^\bot)=\deg T_C(-\log S_C).
\end{align*}
Hence, we obtain a set of inequalities 
\begin{eqnarray*}
\mu(N_{C/Y}^{2}/N_{C/Y}^1) & \le & 0, \\
\mu(N_{C/Y}^{1}/N_{C/Y}^0) & \le & \frac{1}{2} \deg T_C(-\log S_C), \\  
\mu(N_{C/Y}^{0})           & \le & \deg T_C(-\log S_C). 
\end{eqnarray*}
Using $\mu=\frac{\deg}{\rm rank}$ and adding all three inequalities we obtain a single inequality
\begin{equation*} \tag{1.4} \label{1.4}
\deg N_{C/Y} \leq \frac{\rk (N^1_{C/Y})+\rk (N^0_{C/Y})}{2} \cdot \deg T_C(-\log S_C).
\end{equation*}
It satisfies equality if and only if (HHP) holds.
\end{definition}

These conditions are called (HHP) since Hirzebruch \cite{hi73}, in part with H\"ofer \cite{bhh87}, has studied embedded curves on ball quotients
and Hilbert modular surfaces and studied proportionality inequalities involving intersection numbers that attain equality if and only if the curve is
the compactification of a Shimura curve. Hirzebruch's inequalities together with our proof of them can also be found in \cite[Thm. 0.1]{mvz08}.

\begin{proposition} \label{prop 1.5}  {~} \\
(i) If $\varphi: C\to Y$ satisfies (HHP), then $\varphi^*T_Y(-\log S_Y)$
is a direct summand of an orthogonal decomposition of $\varphi^*T_{\overline A_g}(-\log S)$ with respect to the Hodge metric.\\
(ii) If $Y^0\subset A_g$ is a special subvariety, then $\varphi^*T_Y(-\log S_Y)$
is a direct summand of an orthogonal decomposition of $\varphi^*T_{\overline A_g}(-\log S)$ with respect to the Hodge metric and
$\varphi: C\to Y$ satisfies (HHP).
\end{proposition}

\begin{proof}  (i) Assuming (HHP), the slope of the sub bundle  
$$
N_{C/Y}^{i+1}/N_{C/Y}^i\subset N_{C/\overline A_g}^{i+1}/N_{C/\overline A_g}^i
$$ 
is equal to the slope of $N_{C/\overline A_g}^{i+1}/N_{C/\overline A_g}^i$. Since $N_{C/\overline A_g}^{i+1}/N_{C/\overline A_g}^i $
is polystable, $N_{C/Y}^{i+1}/N_{C/Y}^i $ is a direct summand of an orthogonal  decomposition of  $N_{C/\overline A_g}^{i+1}/N_{C/\overline A_g}^i $
w.r.t the  Hermitian-Yang-Mills metric, which is the induced Hodge metric on $N_{C/\overline A_g}^{i+1}/N_{C/\overline A_g}^i$. 

{\bf Claim:} The sub bundle $N_{C/Y}^i\subset N_{C/\overline A_g}^i,\quad 0\leq i\leq 2$  is a direct summand and orthogonal.
 
{\sl Proof of the claim.}  For $i=0.$  Since 
$$
N_{C/Y}^{0}/N_{C/Y}^{-1}=N_{C/Y}^{0},\quad N_{C/\overline A_g}^{0}/N_{C/\overline A_g}^{-1}=N_{C/\overline A_g}^{0},
$$
we have shown above 
$$
N_{C/Y}^0\subset N_{C/\overline A_g}^0
$$ 
is a direct summand and  of an orthogonal decomposition of $N_{C/\overline A_g}^0 $ w.r.t.the Hodge metric.  
Let $p: N_{C/\overline A_g}^0\to  N_{C/Y}^0 $ denote the projection.
  
For $i=1$, we consider the following commutative diagramm

$$
\xymatrix{0 \longrightarrow  N_{C/Y}^0 \ar[r]  \ar@{^{(}->}[d]  & N_{C/Y}^1  \ar[r]  \ar@{^{(}->}[d]
& N_{C/Y}^1/N_{C/Y}^0 \ar@{^{(}->}[d]  \longrightarrow 0  \\
0 \longrightarrow  N_{C/\overline A_g}^0\ar[r]                 & N_{C/\overline A_g}^1 \ar@{=}[d] \ar[r] &N_{C/\overline A_g}^1/N_{C/\overline A_g}^0   \longrightarrow 0 \\
& \makebox[1cm][l]{$ N_{C/\overline A_g}^0  \ar[d]^{(p,0)}\oplus N_{C/\overline A_g}^1$}&\\
& N_{C/\overline A_g}^0 & } 
$$

Since the composition map 
$$   
N_{C/Y}^0\to N_{C/Y}^1 \to  N_{C/\overline A_g}^1\to  N_{C/\overline A_g}^0 \stackrel{p}{\rightarrow}N_{C/Y}^0
$$
is the identity, the short exact sequence  
$$
0\to N_{C/Y}^0\to N_{C/Y}^1 \to N_{C/Y}^1/N_{C/Y}^0 \to 0
$$
splits, and 
$$   
N_{C/Y}^1 = N_{C/Y}^0\oplus N_{C/Y}^1/N_{C/Y}^0\subset N_{C/\overline A_g}^0 \oplus N_{C/\overline A_g}^1/N_{C/\overline A_g}^0=N^1_{C/\overline A_g} .
$$
Since $N_{C/Y}^0\subset N_{C/\overline A_g}^0 $ and $N_{C/Y}^1/N_{C/Y}^0\subset N_{C/\overline A_g}^1/N_{C/\overline A_g}^0$ are 
direct summands  and orthogonal, $N_{C/Y}^1\subset N^1_{C/\overline A_g} $ is a direct summand and orthogonal.\\[.1cm]
Finally, replacing $N^0$ by $N^1/N^0$, $N^1$ by $N^2$  and    $ N^1/N^0$ by $N^2/N^1$ in the above diagramm, we obtain
$N_{C/Y}\subset N_{C/\overline A_g}$ is a direct summand and orthogonal. The claim is thus proven. 

We are now in the position to finish i).  Since by diagram (\ref{1.3}) 
$$  
\varphi^* T_Y(-\log S_Y)=T_C(-\log S_C)\oplus N_{C/Y}\subset T_C(-\log S_C)\oplus N_{C/\overline A_g}=(i\circ\varphi)^*T_{\overline A_g}(-\log S),
$$
and by the above claim  $N_{C/Y}\subset N_{C/\overline A_g}$ is a direct summand and orthogonal,  
$$ 
\varphi^* T_Y(-\log S_Y)\subset (i\circ\varphi)^*T_{\overline A_g}(-\log S)
$$
The proof of i) is thus complete. 

(ii)  Let $i: Y^0 \hookrightarrow A_g$ be a special subvariety. Then $Y^0$ is a locally symmetric subvariety of the locally symmetric varity $A_g$  and the vector 
subbundle  $di: T_{Y^0}\hookrightarrow i^* T_{A_g}$  is a locally homogenous subbundle of the locally homogenous bundle  $i^*T_{A_g}$ 
in the sense of Mumford \cite[Sect. 3]{mum}.
As a locally homogenous bundle can be decomposed as direct sum of irreducible locally homogenous  subbundles and this decomposition is orthogonal w.r.t.
the invariant metric, $di: T_{Y^0}\hookrightarrow i^* T_{A_g}$  is a direct summand and orthogonal.   
Note that the Deligne extension of the sheaf of differential 1-forms  is the sheaf of differential 1-forms 
with logarithmic poles at infinity. By the uniqueness of Deligne's extension we get that
$di: T_{Y}(-\log S_Y)\hookrightarrow i^* T_{\overline A_g}(-\log S)$ is a direct summand and orthogonal.   
Thus,
$$ 
di: \varphi^*T_{Y}(-\log S_Y)\subset(i\circ \varphi)^*T_{\overline A_g}(-\log S)
$$
is a direct summand and orthogonal. (The argument here was pointed out by the referee.)\\
Since $\varphi: C^0\to A_g$ is a morphism of Shimura varieties, one has the decomposition
$$ 
  (i\circ\varphi)^*T_{\overline A_g}(-\log S)\simeq S^2(E^{0,1})=\sL^{-2}\otimes S^2(\sT)\oplus\sL^{-1}\otimes \sT\otimes\sU^\vee\oplus S^2(\sU^\vee)
$$
of polystable subbundles which can be decomposed further as the direct sum of irreducible stable subbundles.
$$
(i\circ\varphi)^*T_{\overline A_g}(-\log S)=K_1\oplus\cdots \oplus K_l.
$$ 
By a theorem of Atiyah \cite{atiyah} the category of vector bundles over any compact complex manifold
is Krull-Schmidt, i.e., in our case if there is a second decomposition
$$ 
(i\circ\varphi)^*T_{\overline A_g}(-\log S)=K_1'\oplus\cdots\oplus K_{l'}'
$$
of irreducible subbundles, then up to a permutation one has
$$
K_i\simeq K_{i'}'.
$$
This shows that $(i\circ\varphi)^*T_{Y}(-\log S_Y)$ is the direct sum of some direct factors of 
$\sL^{-2}\otimes S^2(\sT)$,  $\sL^{-1}\otimes \sT\otimes\sU^\vee$ and $S^2(\sU^\vee)$
and therefore the relative proportionality inequality \eqref{1.4} is an equality. 
\end{proof} 

In the proof of  Proposition 1.5 we see that the inclusion 
$$\varphi^*T_Y(-\log S_Y)\subset  (i\circ \varphi)^*T_{\overline A_g}(-\log S)$$
 is compatible with
the decompositions
$$
\begin{array}{cccc}
\varphi^*T_Y(-\log S_Y) =  &T_C(-\log S_C) \oplus  N_{C/Y} =& T_C(-\log S_C) \oplus &\bigoplus_{i=0}^1  N^{i+1}_{C/Y}/N^i_{C/Y}\\
\raisebox{2ex}{\rotatebox{-90}{$\hookrightarrow$}}                        &   || \quad  \quad \quad  \quad \quad   \raisebox{2ex}{\rotatebox{-90}{$\hookrightarrow$}}            & || &   \raisebox{2ex}{\rotatebox{-90}{$\hookrightarrow$}}  \\
(i\circ\varphi)^*T_{\overline A_g}(-\log S)=  &T_C(-\log S_C)  \oplus  N_{C/\overline A_g} 
=& T_C(-\log S_C) \oplus &\bigoplus_{i=0}^1  N^{i+1}_{C/\overline A_g}/N^i_{C/\overline A_g}.
\end{array}
$$
\begin{example}  If $Y$ is a Shimura surface, then $N_{C/Y}$ is a line bundle and there are three cases in which we write 
(HHP) in terms of more familiar intersection numbers, see \cite{bhh87} and \cite[Thm. 0.1]{mvz08}:  \\
(i) $Y$ is a \emph{Hilbert modular surface}:
$$ 
N_{C/Y}=N^0_{C/Y} \subset \sL^{-2}\otimes S^2(\sT)/\sL^{-2},\quad (HHP):\quad \omega_Y(S)\cdot C+2C^2=0.
$$
(ii) $Y$ is a \emph{Picard modular surface}:
$$
N_{C/Y} \cong N^1_{C/Y}/N^0_{C/Y} \subset \sL^{-1}\otimes \sT\otimes \sU^\vee,\quad (HHP):\quad  \omega_Y(S)\cdot C+3C^2=0.
$$
(iii) $Y$ is \emph{product of two Shimura curves}:
$$
N_{C/Y} \cong N^2_{C/Y}/N^1_{C/Y}\subset  S^2(\sU^\vee),\quad (HHP):\quad C^2=0,
$$
and $C$ lies in the fibres of one of the projections.
\end{example}

\begin{question} Does the (HHP) for a single compactified Shimura curve $C$ together with
$\varphi: C\to Y\subset \overline A_g$ as above imply that $Y^0\subset A_g$ is a special subvariety, 
if we assume that the algebraic monodromy group $H(Y^0)$ (see introduction) is $\Q$-simple ?
\end{question}

This question seems to be very optimistic and at the same time difficult to answer. 
However, we are not aware of any counterexamples if $H(Y^0)$ is $\Q$-simple.\\

It is our goal in the rest of the paper to show that the existence of \emph{''many''} special curves, e.g. a dense subset of such satisfying (HHP) 
forces $Y^0$ to be a special subvariety. 

\begin{remark} Consider the same situation $\varphi: C\to Y \subset \overline A_g$, where $Y^0\subset A_g$ is a special subvariety and
$C$ an arbitrary curve, not necessarily Shimura. Then one obtains an inequality opposite to \eqref{1.3}, see for example 
\cite[Thm. 0.3 and Thm. 2.3]{mvz08}.
\end{remark}

\section{Thickening of the Higgs field}

We use the same notation as in the previous section. In particular $C$ is a compactified Shimura curve together with a non-constant morphism
$\varphi: C \to  Y \subset \overline A_g$ factoring over a smooth projective subvariety $Y$ such that $S=\partial \overline A_g$ intersects $Y$ and 
the image of $C$ transversely. In the previous section we showed that under these assumptions there is a canonical splitting 
$$
\varphi^*\Omega^1_Y(\log S_Y) \cong \Omega^1_C(\log S_C)\oplus N_{C/Y}^\vee,
$$ 
see \eqref{1.3}. Also we denote by $E=E^{1,0} \oplus E^{0,1}$ the (logarithmic) Higgs bundle on $\overline A_g$ associated to the local system $\V_\C=R^1f_*\C$, 
where $f: X \to A_g$ is the universal family over $A_g$. Its restriction to $Y$ or $C$ will be denoted by the same symbol. 
We also make use of the complex vector bundle $\sV:=\V \otimes {\mathcal O}_{A_g}$ or its restrictions to $Y^0$ and $C^0$. 
The following definition is new in the literature and goes back to our discussions with Viehweg. It enables us to include the normal 
direction to $C$ in $Y$ into our considerations. 

\begin{definition}
We define the \emph{thickening} of the Higgs field $\theta$ on $C$ \emph{in the normal direction} $N_{C/Y}$ as the pullback of the Higgs bundle on 
$Y$ via $\varphi:C\to Y$:
$$
\theta_{C/Y}:=\varphi^*\theta: E^{1,0}\to E^{0,1}\otimes\varphi^*\Omega^1_Y(\log S_Y)=E^{0,1}\otimes(\Omega^1_C(\log S_C)\oplus N_{C/Y}^\vee).
$$
In the same way we define the \emph{thickening} of the Higgs field \emph{in a point} $p\in C$ in the normal direction  $N_{p/Y}$ as
$$
\theta_{p/Y}:=\theta_{C/Y}|_p: E^{1,0}|_p\to  E^{0,1}|_p\otimes \varphi^*\Omega^1_Y(\log S_Y)|_p=
E^{0,1}|_p\otimes(\Omega^1_C(\log S_C)|_p\oplus N_{C/Y}^\vee|_p).
$$
\end{definition}

Consider the $k$-fold tensor product $(E,\theta)^{\otimes k}$ of the Higgs bundle $(E,\theta)$ on $Y$. 
It decomposes as a direct sum
$$
E^{\otimes k}=\bigoplus_{p+q=k}E^{p,q}
$$
where  
$$
E^{p,q}=\bigoplus E^{p_1,q_1}\otimes\cdots\otimes E^{p_k,q_k}
$$
and where the sum ranges over $p_i+q_i=1,\, \sum_{i=1}^k p_i=p,\,\sum_{i=1}^k q_i=q$. 
The Higgs field, again denoted by $\theta,$ decomposes as
$$
\theta: E^{p,q}\to E^{p-1,q+1}\otimes\Omega^1_Y(\log S_Y),
$$
where
$$ 
E^{p_1,q_1}\otimes\cdots\otimes E^{p_k,q_k}
\stackrel{\theta} {\longrightarrow}\bigoplus_{i=1}^k
E^{p_1,q_1}\otimes\cdots \otimes E^{p_i-1,q_i+1}\otimes\cdots\otimes
E^{p_k,q_k}\otimes\Omega^1_Y(\log S_Y)
$$ 
satisfies the Leibniz rule
$$  
\theta|_{E^{p_1,q_1}\otimes\cdots\otimes E^{p_k,q_k}}=\sum_{i=1}^k{\rm id}\otimes\cdots\otimes\theta_{1,0}\otimes\cdots\otimes{\rm id}. 
$$
In the same way as in the definition above we define the thickening $(E,\theta_{C/Y})^{\otimes k}$ and 
$(E,\theta_{p/Y})^{\otimes k}$.

Assume for a moment that $Y^0$ is a locally symmetric quotient of a bounded symmetric domain. Then it is
well-known that $\Omega^1_{Y^0},$ and all Hodge bundles $E^{p,q}_{Y^0}$ are locally homogeneous vector bundles in the sense of Mumford \cite[Sect. 3]{mum}.
Furthermore, the Higgs map $\theta^{p,q}: E^{p,q}_{Y^0}\to E^{p-1,q+1}|_{Y^0}\otimes\Omega^1_{Y^0}$ is an equivariant morphism between
locally homogeneous vector bundles. We decompose $E^{p,q}_{Y^0}$ as the direct sum of irreducible locally homogeneous
subbundles
$$
E^{p,q}_{Y^0}=\bigoplus_i E^{p,q}_{Y^0,i}.
$$
Then we take Mumford's canonical extensions $E^{p,q}_{Y,i}$ \cite[Sect. 3]{mum}, which agrees with Deligne's extension by
\cite[Lemma 2.4]{mvz07} of those irreducible locally homogeneous subbundles, and we use the same symbols
$$ 
E^{p,q}_Y=\bigoplus_i E^{p,q}_{Y,i}
$$
for the extended decomposition by uniqueness of good extensions. Complex conjugation
$$
\V_{\C}\stackrel{-}{\rightarrow}\V_{\C}
$$
induces also a complex conjugation on the Deligne extensions of $\sV|_{C^0}=\V\otimes \sO_{C^0}$ and sends $E^{p,q}_Y$ to $E^{q,p}_Y$
($\simeq E^{p,q\vee}_Y$), hence $E^{p,q}_{Y,i}$ to $E^{q,p}_{Y,\overline i}$ $(\simeq E^{p,q\vee}_{Y,\,i})$.

Given a base point $y\in Y$ ($y$ could lie on the boundary $S_Y$) we consider  
$$
\theta^{p,p}_{y\in Y}: E^{p,p}_y\to E^{p-1,p+1}_y\otimes\Omega^1_Y(\log S_Y)_y,
$$
where $E_y^{p,p}$ carries the induced real structure from $\sV^{\otimes k}$. 
Its real structure is induced from $\V_\R=\V_\Q \otimes \R$. 


\begin{definition} Let $y\in Y$ be a base point: \\ 
(a) $W_{y\in Y}:=\{t\in  E^{p,p}_{y}\,|\, \theta_{y\in Y}(t)=0\}.$\\[.1cm]
(b) A tensor $t\in E^{p,p}_y\cap \sV^{\otimes k}_{\R, y}$ is called a \emph{real Hodge tensor} at the base point $y\in Y$. \\
(c) $W_{y\in Y,\R}:=\{t\in  E^{p,p}_{y}\cap \sV^{\otimes k}_{\R,y}\,|\, \theta_{y\in Y}(t)=0\}$.
\end{definition} 

\begin{remark} It is clear that  $W_{y\in Y,\R}\otimes \C\subset W_{y\in Y}$, but in general they are not equal.
\end{remark}

\begin{proposition}[Parallel Transport] \label{2.4} There exists a unitary subsystem
$\W \subset \V^{\otimes k}$ of pure Hodge type $(p,p)$, which naturally extends to $Y$, carries an induced real structure
from $\V^{\otimes k}$ and such that $\W_{\R,y}=W_{y\in Y,\R}$ for all $y\in Y$.
\end{proposition}

\begin{proof} From the above discussion we know that $E^{p,p}_Y$ decomposes as direct sum of good extensions of irreducible locally homogeneous subbundles
$$ 
E^{p,p}_Y=\bigoplus_i E^{p,p}_{Y,i}.
$$ 
Since $\theta^{p,p} : E^{p,p}\to E^{p-1,p+1}\otimes\Omega^1_Y(\log S_Y)$ is a morphism between good extensions of locally homogeneous vector bundles, 
${\rm ker}(\theta^{p,p})$ is again a direct sum of good extensions of locally homogeneous subvector bundles, which are Higgs subbundles 
(with trivial Higgs field) of $(E,\theta)^{\otimes k}$. By Simpson's polystability all of them have non-positive slopes. We decompose (as holomorphic vector bundles)
$$
{\rm ker}(\theta^{p,p})= {\rm ker}(\theta^{p,p})_0\oplus {\rm ker}(\theta^{p,p})_{<0},
$$
where ${\rm ker}(\theta^{p,p})_0$ is the direct sum of good extensions of locally homogenous subvector bundles of slope zero and
${\rm ker}(\theta^{p,p})_{<0}$ is the  direct sum of good extensions of locally homogeneous subvector bundles of negative slopes.
By Simpson's correspondence ${\rm ker}(\theta^{p,p})_0$ underlies a unitary local subsystem $\W\subset \V^{\otimes k}$ of type $(p,p)$.
${\rm ker}(\theta^{p,p})_0$ is invariant under complex conjugation induced by the real structure on $\V^{\otimes k}.$  This can be seen as follows:
the complex conjugate $\overline{{\rm ker}(\theta^{p,p})_0}$ corresponds to the complex conjugate $\overline{\W}$ of $\W$, which is again a unitary sublocal system
of type $(p,p)$. Hence it vanishes under the Higgs field $\theta^{p,p}$, i.e., $\overline{{\rm ker}(\theta^{p,p})_0}\subset {\rm ker}(\theta^{p,p})$.  
Note that $\overline{{\rm ker}(\theta^{p,p})_0}$ is again the direct sum of some good extensions of locally homogeneuos subvector bundles of slope zero, 
hence we obtain $\overline{{\rm ker}(\theta^{p,p})_0}\subset {\rm ker}(\theta^{p,p})_0$. 
Clearly for all real vector $t\in {\rm ker}(\theta^{p,p})_{0,y}$ we have $\theta^{p,p}_y(t)=0,$ so $t\in  W_{y\in Y,\R}$. 
Conversely, let $t\in W_{y\in Y,\R}$.  Then $t$ is a real vector in ${\rm ker}(\theta^{p,p})_y$. There are no vectors in 
${\rm ker}(\theta^{p,p})_{<0,y}$ fixed by complex conjugation, because complex conjugation takes negative slopes to positive slopes. Therefore 
$t$ is a real vector in ${\rm ker}(\theta^{p,p})_{0,y}$. Thus we have shown
$$
\W_{\R, y} = W_{y\in Y,\R}.
$$ 
Since $\V^{\otimes k}$ has unipotent local monodromies around $S_Y$ and $\W$ is unitary, the local monodromies of $\W$ are trivial around $S_Y$.
Hence $\W$ extends across $S_Y$. 
\end{proof}

\begin{remark} For a rational Hodge tensor $t\in W_{y\in Y}$, $t$ is contained in a unitary local subsystem with a $\Z-$structure.
Hence the orbit
$$
\{ \rho(\gamma)(t)\,|\,\gamma\in \pi_1(Y^0,y)\}
$$
is finite.
\end{remark}

For a Shimura curve $C^0$ mapping to $A_g$ via $\varphi$ as above 
we can describe the above decomposition more precisely. For the Higgs bundle $E$ one has from section 1 
$$
E^{1,0}=\sL\otimes \sT\oplus \sU,\quad E^{0,1}=\sL^{-1}\otimes\sT\oplus \sU^\vee,
$$
where $\sL\otimes \sT$ and $\sL^{-1}\otimes\sT$ are polystable of slopes $\deg\sL$ and $-\deg\sL$ respectively.
Moreover if $\sU\not=0$ then $\sU$ and $\sU^\vee$ are both polystable of slope zero. So one obtains immediately:

\begin{lemma}  The sheaves $E^{p,q}$ are direct sums of polystable
sheaves $E^{p,q}_\iota$ of slopes $\mu(E^{p,q}_\iota)=\iota\deg\sL$ and one has: \\
(a) If $\sU=0,$ then $E^{p,q}_\iota\not=0$ if and only if $\iota=p-q,$ and $E^{p,q}=E_\iota^{p,q}.$\\
(b) If $\sU\not=0,$ then $E^{p,q}_\iota\not=0$ if and only if $\iota\in\{-q,\ldots, p\}.$ In this case $E^{p,q}_\iota$ is a direct
sum of sheaves of the form
$$
\bigoplus_{m-l=\iota}(\sL\otimes\sT)^{\otimes m}\otimes(\sL^{-1}\otimes\sT)^{\otimes l}\otimes\sU^{\otimes(p-m)}\otimes\sU^{\vee\otimes(q-l)}.
$$
(c) The sheaf $E^{p,p}_\iota$ is dual to $E^{p,p}_{-\iota}$.
\end{lemma}

Consider the decomposition \eqref{1.3}
$$ 
\varphi^*T_Y(-\log S_Y)=N_{C/Y}\oplus T_C(-\log S_C).
$$
The assumption that $\varphi: C\to Y$ satisfies (HHP) implies the decomposition  in the proof of Proposition 1.2
$$ 
N_{C/Y}=N^0_{C/Y}\oplus N^1_{C/Y}/N^0_{C/Y}\oplus N^2_{C/Y}/N^1_{C/Y}
$$ 
such that 
$$
N^0_{C/Y}\subset N^0_{C/\overline A_g}=\sL^{-2}\otimes S^2(\sT)/\sL^{-2} \subset \sH om(\sL\otimes\sT, \sL^{-1}\otimes\sT)/\sL^{-2},
$$
$$
N^1_{C/Y}/N^0_{C/Y}\subset N^1_{C/\overline A_g}/N^0_{C/\overline A_g}=\sL^{-1}\otimes\sT\otimes\sU^\vee=\sH om (\sL\otimes\sT, \sU^\vee)
$$
and
$$
N^2_{C/Y}/N^1_{C/Y}\subset  N^2_{C/\overline A_g}/N^1_{C/\overline A_g}=S^2(\sU^\vee)\subset \sH om(\sU, \sU^\vee)
$$
are direct polystable factors of slopes $-2\deg\sL,$ respectively $-\deg\sL,$ respectively $0$.
In this way  we may decompose the thickening $\theta_{C/Y}$ in the form
$$
\theta_{C/Y}=\theta_C+\theta_{N_{C/Y}}=\theta_C+\theta_{N^0_{C/Y}}+\theta_{N^1_{C/Y}/N^0_{C/Y}}+\theta_{N^2_{C/Y}/N^1_{C/Y}}. 
$$

Using that decomposition we obtain:

\begin{lemma} \label{2.7} The thickening $\theta_{C/Y}$ on $E^{p,q}_\iota$ can be decomposed as a direct sum of morphisms:
$$
E^{p,q}_{\iota} \xrightarrow{\theta_C+\theta_{N^0_{C/Y}}} E^{p-1,q+1}_{\iota-2}\otimes (\Omega^1_C(\log S_C)\oplus N^{0 \; \vee}_{C/Y}),$$
$$
E^{p,q}_{\iota} \xrightarrow{\theta_{N^1_{C/Y}/N^0_{C/Y}}} E^{p-1,q+1}_{\iota-1}\otimes (N^1_{C/Y}/N^0_{C/Y})^\vee
$$
and
$$
E^{p,q}_{\iota} \xrightarrow{\theta_{N^2_{C/Y}/N^1_{C/Y}}} E^{p-1,q+1}_{\iota}\otimes (N^2_{C/Y}/N^1_{C/Y})^\vee.
$$
between polystable sheaves of the same slopes.
\end{lemma}

\begin{proof} Write 
$$
E^{1,0}=\sL\otimes\sT\oplus \sU,\quad E^{0,1}=\sL^{-1}\otimes\sT\oplus \sU^\vee,
$$
then $\theta_{C/Y}:E^{1,0}\to E^{0,1}\otimes\varphi^*\Omega^1_Y(\log S_Y)$ decomposes into the following terms:
$$
\sL\otimes\sT \xrightarrow{\theta_C} \sL^{-1}\otimes\sT\otimes\Omega^1_C(\log S_C),
$$
$$\sL\otimes\sT \xrightarrow{\theta_{N^0_{C/Y}}} \sL^{-1}\otimes\sT\otimes N^{0\vee}_{C/Y},
$$
$$\sL\otimes\sT \xrightarrow{\theta_{N^1_{C/Y}/N^0_{C/Y}}} \sU^\vee\otimes (N^1_{C/Y}/N^0_{C/Y})^\vee
$$
and
$$
\sU \xrightarrow{\theta_{N^2_{C/Y}/N^1_{C/Y}}} \sU^\vee\otimes (N^2_{C/Y}/N^1_{C/Y})^\vee.
$$
This proves the lemma for the case $k=1$. In general, one reduces the cases $k\geq 2$ to the case $k=1$ using the fact that
the thickening $\theta_{C/Y}^{\otimes k}$ is defined by the Leibniz rule. 
\end{proof}
 

\section{Parallel transport of real Hodge tensors on connected cycles of special curves}

In this section let $Y\subset \overline A_g$ be a smooth projective subvariety, which meets $S=\partial \overline A_g$ transversely.
Assume $Y$ contains a connected cycle $\sum_i C_{i}$ of finitely many compactified embedded special curves, such that
each component $C_i$ meets $S_Y=S \cap Y$ transversely and satisfies (HHP). Using base points $y_i\in C_i^0$ and 
notations from the previous section we introduce the following subspaces:
$$
\begin{array}{ccc}
W_{y_i\in Y}&:=&\{t\in  E^{p,p}_{y_i}\,|\, \theta_{y_i\in Y}(t)=0\},     \\ 
\cap &  &                                                                \\
W_{y_i\in C_i}&:=&\{t\in  E^{p,p}_{y_i}\,|\, \theta_{y_i\in C_i}(t)=0\}  \\
\end{array}
$$
and the real spaces
$$
\begin{array}{ccc}
W_{y_i\in Y,\R}&:=&\{t\in  E^{p,p}_{y_i}\cap \sV^{\otimes k}_{\R}|_{y_i}\,|\, \theta_{y_i\in Y}(t)=0\},\\
\cap &  &                                                                                                  \\
W_{y_i\in C_i,\R}&:=&\{t\in  E^{p,p}_{y_i}\cap \sV^{\otimes k}_{\R}|_{y_i}\,|\, \theta_{y_i\in C_i}(t)=0\}.\\
\end{array}
$$
Fixing a base point $y_1\in C^0_1$ we now need to study the parallel transport of real vectors in $W_{y_1\in Y,\R}$ along paths in the connected
subspace $\bigcup_i C_i$.

\begin{proposition} \label{3.1} (a) The real subspace $W_{y_1\in Y,\R} \subset \V^{\otimes k}_{\R}$ is invariant under
the monodromy action $\rho^{\otimes k}(\pi_1(\bigcup_iC^0_i, y_1))$. \\
(b) Assume that $E^{p,p}_{C_i}$ is polystable of slope zero for all $C_i$. Then $W_{y_1\in Y}$ is invariant under
the monodromy action $\rho^{\otimes k}(\pi_1(\bigcup_iC^0_i, y_1))$.
\end{proposition}  
  
\begin{proof} (a) We have the decomposition 
$$ 
\theta_{y_1/Y}=\theta_{y_1/C_1 }\oplus\theta_{N_{C_1/Y},y_1}.
$$
Hence, 
$$ 
W_{y_1\in Y,\R}=\{\, t\in W_{y_1\in C_1,\R}\,|\, \theta_{N_{C_1/Y},y_1}=0\,\}.
$$
By Proposition \ref{2.4} there exists a unitary subsystem $\W_{C_1} \subset \V^{\otimes k}_{\R}$ of Hodge type $(p,p)$ such that
$W_{y_1\in C_1,\R}=\W_{C_1, y_1}$. 

Let $\sW_{C_1}\subset E^{p,p}_{C_1}$ denote the polystable subbundle of slope zero corresponding to $\W_{C_1}$. Then by Lemma \ref{2.7}
$$
\theta_{N_{C_1/Y}}: \sW_{C_1}\to  \theta_{N_{C_1/Y}}(\sW_{C_1}) 
$$ 
is a morphism between polystable bundles of slope zero. Hence the kernel
$$ 
{\rm ker}( \theta_{N_{C_1/Y}}: \sW_{C_1}\to  \theta_{N_{C_1/Y}}( \sW_{C_1} )  )=:\sW'_{C_1}
$$   
is a polystable subbundle of $\sW_{C_1}$ of slope zero. Therefore it underlies a unitary subsystem $\W'_{C_1}\subset \W_{C_1}\otimes \C$.
From the construction of $\W'_{C_1}$ we see that 
$$ 
W_{y_1\in Y,\R}\subset  \W'_{C_1, y_1}.
$$ 
We start with a real vector $t_1 \in W_{y_1\in Y,\R}$ and denote by $t_2$ the parallel transport of $t_1$ as a vector in the fibre of the local system  
$\V^{\otimes k}_{\R,y_1}$ along some path in $C_1^0$ from $y_1$ to $y_2\in C^0_1\cap C^0_2$. 
Since $t_1$ is contained in the fibre of the subsystem $\W'_{C_1}\subset \V^{\otimes k}_{\C,C_1}$  at $y_1$, $t_2$ is a real vector 
(because of the real structure on $\V^{\otimes k}$) and contained in the fibre of $\W'_{C_1}$ at $y_2$. By the construction
of $\W'_{C_1}$, we see that $\theta_{y_2/C_1}(t_2)=0$ and $\theta_{N_{C_1/Y}, y_2}(t_2)=0$, i.e., $t_2\in W_{y_2\in Y,\R}$. 

Regarding $y_2\in C^0_2$ we repeat the above argument and continue the parallel transport of $t_2$ along some path in 
$C_2^0$ from $y_2$ to $y_3\in C_2\cap C_3$ etc.. This shows that $W_{y_1\in Y,\R}\subset \V^{\otimes k}_{\R, y_1}$
is invariant under the monodromy action $\rho^{\otimes k}(\pi_1(\bigcup_iC^0_i, y_1))$.

(b) Since $E_{C_1}^{p,p}$ is polystable of slope zero, ${\rm ker }(\theta^{p,p}_{C_1})$ is a Higgs subbundle of slope zero
(with trivial Higgs field) and it corresponds to an (extended) unitary local system  $\W_{C_1}\subset \V^{\otimes k}_{C_1}$ 
with an induced real structure, and such that $\W_{C_1,y}=W_{y\in C_1}$ for all $y\in C_1$. Since
$W_{y_1 \in Y}=\{t\in W_{y_1\in C_1}\,|\, \theta_{N_{C_1/Y},y_1}(t)=0\,\}$, by the same argument as in (a) we find a unitary subsystem
$\W'_{C_1}\subset \W_{C_1}$ such that $\W'_{C_1, y}=W_{y\in Y}$ for all $y\in C_1$. The rest of the proof is exactly the same as in (a). 
\end{proof}  

\begin{definition}
We say that $Y^0$ can be covered by a smoothing of a cycle $\sum_i a_i C_i$ of special curves $C_i$ satisfying (HHP), if there is 
a suitable linear combination $\sum_i a_iC_i$ of compactified special curves $C_i \subset Y$ satisfying (HHP) which 
can be deformed into a generically smooth family of curves $\cup_{z\in Z} C_z$ filling out $Y$, i.e., such that $\sum_i a_iC_i$ is a degenerate fibre 
of a generically smooth family of curves $\cup_{z\in Z}C_z$ over some parameter scheme $Z$. 
\end{definition}

\begin{proposition} \label{3.2} Assume that the algebraic  monodromy group $H(Y^0)$ of $Y^0$ (defined in the introduction) is $\Q$-simple and that 
$Y^0$ can be covered by a smoothing of a cycle $\sum_i a_i C_i$ of special curves $C_i$ satisfying (HHP). We fix a base point $y_0 \in C^0_1$.\\
(a) Then $W_{y_1\in Y,\R} \subset \V^{\otimes k}_{\R,y_1}$ is $\rho^{\otimes k}(\pi_1(Y^0, y_0))$-invariant.\\
(b) Under the assumption in Prop.~\ref{3.1} (b), $W_{y_0\in Y}\subset \V^{\otimes k}_{\C,y_1}$  is $\rho^{\otimes k}(\pi_1(Y^0,y_0))$-invariant.
\end{proposition}

The following Lemma is Proposition  2.2.2 in \cite{zuo99}. It will be used below. 

\begin{lemma} \label{zuoslemma} 
Let $X$ be a smooth complex quasi-projective variety, $k$ a field of characteristic 0, $G$ an almost simple $k$-algebraic group and 
$$
\rho: \pi_1(X, *)\to G(k)
$$
be a Zariski dense representation. Then the following holds: \\[.1cm]
{(1)} If $\pi: X'\to X$ is a surjective and generically finite morphism, and $X'$ is smooth, then 
$\pi^*(\rho)$ is again Zariski dense.\\
{(2)}  If $f: X\to Y$ is a surjective morphism to a smooth quasi-projective variety $Y$ with connected  
fibres, and if $f^{-1}(y)\subset X$ is a smooth fibre, then there are two possibilities:\\
either\\
{(i)} the restriction $\rho|_{f^{-1}(y)}$ is again Zariski dense, or\\[.1cm]
{(ii)} $\rho|_{f^{-1}(y)}$ has finite image.
\end{lemma}

\begin{proof} (Proposition \ref{3.2})  (a) Fix a smooth curve $C_z$ in the family $\cup_{z\in Z}C_z$ and a base point $* \in C_z$. Then $C_z^0$ deformes to 
$\sum_ia_iC_i^0$ and $*$ moves to $y_0\in \sum_iC_i^0$ along a path $\gamma_{* y_0}$. This implies that any loop lying
on $C_z^0$ with base point $*$ is homotopic to some loop lying on $\sum_iC_i^0$ with base point $y_0$. 
By Proposition \ref{3.1} the induced representation
$$
\rho_{C_z^0}: \pi_1(C_z^0,*)\to \pi_1(Y^0,*)\stackrel{\rho^{\otimes k}}{\rightarrow}\V_*^{\otimes k}
$$
stabilizes the real subspace of $W_*\subset E^{p,p}_{\R,*}$ which is the parallel transport of $W_{y_0\in Y}$ along the path $\gamma_{*y_o}^{-1}$.

By assumption, the algebraic monodromy group $H(Y^0)$ is $\Q-$simple. The covering family is given by a correspondence in $Y \times Z$ 
and can be chosen such that there are 
finitely many curves through a generic point of $Y$. Therefore, after taking a generically finite base change  $Y^{0'} \to Y^0$, 
we may assume that the family gives rise to a surjective map $g: Y^0\to Z^0$ with connected fibres, and such that $C_z^0\subset Y^0$ 
is a smooth fibre of $g$. 
Note that this modification does not change the algebraic monodromy group $H(Y^0)$ by (1) in Lemma \ref{zuoslemma}.

By (2) in Lemma \ref{zuoslemma} there are two possibilities:
either (i): $H(C_z^0)=H(Y^0)$, or (ii): $H(C_z^0)$ is a finite group. 
The case (ii) is impossible. Otherwise the restricted representation 
$\rho_{C^0_z}: \pi_1(C^0_z*)\to Sp(2g,\mathbb Q)$ would have finite image, 
which implies that the restricted period map $\varphi :C^0_z\to A_g$ is constant. A contradiction.\\[.1cm]
So we obtain $H(C_z^0)= H(Y^0)$. Since $\rho^{\otimes k}_{C_z^0}$ stabilizes $W_*$, which is indeed an algebraic condition for the monodromy matrices of $\rho_{C^0_z}$, 
the representation $H(C_z^0)$ in $\V_*^{\otimes k}$ also stabilizes $W_*$. This shows that $\rho^{\otimes k}(\pi_1(Y^0,*))$ stabilizes $W_*$.
Now by moving the base point $*$ along the path $\gamma_{* y_0}$ to $y_0$ we obtain that $W_{y_0\in Y,\R}$ is $\rho^{\otimes k}(\pi_1(Y^0,y_0))$-invariant.
The proof of (b) is the same as the one of (a). 
\end{proof} 

\begin{corollary} \label{3.3} 
Assume $W_{y_0\in Y}$ carries a $\Q$-structure from $\V^{\otimes k}$. Then the subsystem $\U^{p,p}_{Y^0}$ has finite monodromy  
and $W_{y_0\in Y}$ extends to a subspace of sections of $\V^{\otimes k}_{Y^0}$.
\end{corollary} 

We assume now $Y^0$ is contained in a Shimura subvariety $M^0 \subset A_g$ of type $SO(2,n)$ with toroidal compactification $M \subset \overline A_g$. 
Then $\V^{\otimes 2}_{M^0}$ contains a sub-VHS of Hodge structures $\V'$, whose corresponding Higgs bundle has the form
$$ 
E=E^{2,0}\oplus E^{1,1}\oplus E^{0,2},\quad \quad  \theta^{2,0}: T_{M}(-\log S_M)\otimes E^{2,0}\simeq E^{1,1},\quad \theta^{1,1}=\theta^{2,0\vee}.
$$
The \emph{Griffiths-Yukawa coupling} for $E$ along any subvariety $Z \subset {M}$ meeting $S_M$ transversely is the iterated Kodaira-Spencer derivative 
$$
E^{2,0} \longrightarrow E^{0,2} \otimes S^2 \Omega^1_Z(\log S_Z).
$$
The following statements and proofs will use this notation.

\begin{theorem} \label{3.4} Let $Y \subset \overline A_g$ be as above. Assume that $Y^0$ is contained in a Shimura subvariety $M^0 \subset A_g$ of type $SO(2,n)$
We assume that $Y$ and $M$ intersect the boundary $S$ of $A_g$ transversely, and that 
$Y$ can be covered by a smoothing of a cycle $\sum_i a_i C_i$ of special curves $C_i \subset Y$ satisfying (HHP). Then: \\ 
(a) If $W_{y_0\in Y}=W_{y_0\in Y,\R}\otimes \C$ for some $y_0\in C_1$ then $Y^0\subset M^0$ is a special subvariety of orthogonal type.\\
(b) If the Griffiths-Yukawa couplings along all $C_i$ do not vanish then $Y^0\subset M^0$ is a special subvariety of orthogonal type. \\
(c) If the Griffiths-Yukawa coupling along $Y$ vanishes then $Y^0\subset M^0$ is a special subvariety of unitary type, i.e., a ball quotient.
\end{theorem}

\begin{remark} It is not hard to show that the assumptions of the theorem 
are necessary, since by Borcherds' results \cite{borcherds} any Shimura variety of type $SO(2,n)$ contains 
sections of powers of automorphic line bundles which are unions of orthogonal special subvarieties and components of $S_M$.

One can show that $\Omega^1_Y(\log S_Y)$ is nef on $Y$, and $\omega_Y(S_Y)$ is ample with respect to $Y^0$. This follows from our transversality assumptions. 

In the assertions (a) and (b) of the theorem one may replace the assumption on the smoothing of the cycle $\sum_i a_i C_i$ of special curves 
$C_i \subset Y$ satisfying (HHP) by the following: Assume that there is a connected union $C_1 \cup \cdots \cup C_l$ of special 
curves satisfying (HHP) and such that the image of 
$$
\pi_1(\bigcup C_i^0,*) \longrightarrow \pi_1(Y^0,*)
$$ 
has finite index for some basepoint $*$.
\end{remark}

\begin{proof} $Y^0$ is contained in $M^0$, which is a Shimura variety for $SO(2,n)$ without compact factors. All Hermitian type 
subgroups of $SO(2,n)$ except $G=\SL_2 \times \SL_2$ and quaternionic versions (see main theorem in \cite{sz}) are $\Q$-simple for rank reasons 
and either orthogonal or unitary. Hence $H(Y^0)$ will be $\Q$-simple unless $\dim(Y)=2$. In that case it follows that $Y^0$ is uniformized by 
a product $\H \times \H$ of upper half planes. For the rest of the proof we may therefore assume that $H(Y^0)$ 
is $\Q$-simple and $\dim(Y) \ge 3$. \\
(a) By Prop. \ref{3.2} (a) the real subspace $W_{y_0\in Y,\R}$ is $\rho^{\otimes 2}(\pi_1(Y^0, y_0))$-invariant.
Hence $W_{y_0\in Y,\R}\otimes \C$ defines a unitary subsystem $\U$ of the local system $\V'_{Y,\C}$ underlying the Higgs bundle $E$ (see introduction) 
and a corresponding decomposition of Higgs bundles
$$
(E^{2,0}_{Y}\oplus E^{1,1'}_{Y}\oplus E^{0,2}_{Y},\theta_{Y})\oplus (E^{1,1''}_{Y},0).
$$
Note that $\theta_Y: E^{1,1}_{Y} \to E^{0,2}_{Y} \otimes \Omega^1_Y(\log S_Y)$ is surjective, since the pair $(Y,S_Y)$ is transversely 
embedded in $(\overline{A_g},S)$. Therefore we have $\rk E^{1,1'}_{Y}=\dim Y$, and 
$$
\rk E^{1,1''}_{Y}=\dim  W_{y_0\in Y,\R}\otimes \C=\dim W_{y_0\in Y}=\dim M-\dim Y.
$$
This implies that 
$$
\theta_{Y}: T_Y(-\log S_Y)\otimes E^{2,0}_Y\to E^{1,1'}_Y
$$ 
is an isomorphism. Hence the image of $Y^0$ in $A_g$ is a locally symmetric quotient of the period domain $D$ of orthogonal type associated to 
the complement $\U^\perp$ of $\U$ in $\V'$. As a consequence, $Y^0\hookrightarrow M^0$ is a totally geodesic embedding. 
Together with the rigidity of $Y^0\subset M^0$, which follows from $\dim(Y) \ge 2$ \cite[Lemma 1.5]{mvz08}, 
we obtain that $Y^0\subset M^0$ is a special subvariety of orthogonal type by the arguments in loc. cit.. \\
(b) We will give two proofs. First Proof: The non-vanishing of the Griffiths-Yukawa coupling along $C_i$ implies that 
$$
\V'_{C_i}=\V''\oplus \U^{1,1},
$$
where $\V''$ is a sub-VHS with rank one Hodge bundles 
$$
E^{2,0}_{C_i}\oplus E^{1,1'}_{C_i}\oplus E^{0,2}_{C_i}:=E^{2,0}_{C_i}\oplus \theta(E^{2,0}_{C_i}) \oplus \theta^2(E^{2,0}_{C_i}),
$$ 
and $U^{1,1}_{C_i}$ 
is a sub-VHS of pure Hodge type $(1,1)$. Hence $E^{1,1}_{C_i}$ is polystable of slope zero. Fix a base point $y_0\in C^0_1$. Then by 
Proposition \ref{3.2} (b) the subspace $W_{y_0\in Y}\subset \V'_{\C,y_0}$ is $\rho^{\otimes 2}(\pi_1(Y^0, y_0))$-invariant. 
Hence the Higgs bundle $E$ associated to $\V'_{Y^0}$ decomposes as 
$$ 
(E^{2,0}_{Y}\oplus E^{1,1'}_{Y}\oplus E^{0,2}_{Y},\theta_{Y})\oplus (E^{1,1''}_{Y},0),
$$
where the Higgs subbundle $(E^{1,1''}_{Y},0)$ corresponds to the unitary subsystem of rank equal to $\dim M-\dim Y$ defined above.
This implies that 
$$
\theta_{Y}: T_Y(-\log S_Y)\otimes E^{2,0}_Y\to E^{1,1'}_Y
$$ 
is an isomorphism. As in (a), $Y^0\subset M^0$ is a totally geodesic embedding and  
the rigidity of $Y^0\subset M^0$ implies that $Y^0\subset M^0$ is a special subvariety of orthogonal type.\\
Second proof for (b): Let $(F_Y,\theta_Y)\subset (E^{2,0}_Y\oplus E^{1,1}_Y\oplus E^{0,2}_Y,\theta_Y)$ denote the unique saturated Higgs subsheaf generated
by $E^{2,0}_Y$ and $\theta_Y.$ Then $F_Y^{2,0}=E^{2,0}_Y$ and $ F^{0,2}_Y=E^{0,2}_Y,$ since the Griffiths-Yukawa coupling does not vanish.
The non-vanishing of the Griffiths-Yukawa coupling along $C_i$ implies that 
$$
\V'_{C_i}=\V''\oplus \U^{1,1},
$$
where $\V''$ is a sub-VHS with rank one Hodge bundles
$$
E^{2,0}_{C_i}\oplus E^{1,1'}_{C_i}\oplus E^{2,0}_{C_i}
$$
as above and $U^{1,1}_{C_i}$ is a sub-VHS of pure Hodge type $(1,1)$. Using condition (HHP) for $C_i\subset Y$ we see that 
$$
\begin{array}{rcl}
F^{1,1}_Y|_{C_i} & = & \theta_{C_i}(T_{C_i}(-\log S_{C_i})\otimes E^{2,0}_{C_i})\oplus \theta_{N_{C_i/Y}}(N_{C_i/Y}\otimes E^{2,0}_{C_i}) \\
& = & E^{1,1'}_{C_i}\oplus \theta_{N_{C_i/Y}}(N_{C_i/Y}\otimes E^{2,0}_{C_i}),
\end{array}
$$
where $\theta_{N_{C_i/Y}}(N_{C_i/Y}\otimes E^{2,0}_{C_i})$ is a direct factor of $U^{1,1}_{C_i}$.  In particular
$\det(F^{1,1}_Y)\cdot C_i=0.$  Hence, $\det(F^{1,1}_Y) \cdot C_z=0,$ where $C_z$  is a smooth curve in the family 
$\cup_{z\in Z}C_z$ and meets $S_Y$ transversely.
Note that $\deg F_{C_z}=0$, and by Simpson's polystability for the logarithmic Higgs subsheaf 
$F_{C_z},\theta_{C_z}\subset( E^{2,0}_{C_z}\oplus E^{1,1}_{C_z}\oplus E^{0,2}_{C_z},\theta_{C_z})$ 
we obtain a corresponding sub-VHS $\mathbb V^{'''}_{C_z}\subset \mathbb V'_{C_z}$.  Since 
$H(Y^0)$ is $\Q$-simple, the same argument as in the first proof of (b) shows that $\mathbb V^{'''}_{C_z}$ 
extends to a sub-VHS over $Y^0,$ which uniformizes $Y^0$ as a special subvariety of orthogonal type.\\
(c) The vanishing of the Griffiths-Yukawa coupling on $Y$ implies that the Higgs subsheaf generated by $E^{2,0}_Y$ and $\theta_Y$ has the form
$$
(F_Y,\theta_Y)=(E^{2,0}_Y\oplus \theta_Y(T_Y(-\log S_Y)\otimes E^{2,0}_Y),\theta_Y).
$$
Therefore one has 
$$ 
\theta_Y(T_Y(-\log S_Y)\otimes E^{2,0}_Y)\otimes\sO_{C_i}=\theta_{C_i}(T_{C_i}(-\log S_{C_i})\otimes E^{2,0}_{C_i})\oplus \theta_{N_{C_i/Y}}(N_{C_i/Y}\otimes E^{2,0}_{C_i}).
$$
Note that in this case
$$
E^{2,0}_{C_i}=\sL,\quad E^{1,1}_{C_i}=\sL^{-1}\oplus \sL\oplus \sU^{1,1},
$$
where $\sU^{1,1}$ is polystable of degree zero and such that 
$$
\theta_{C_i}: T_{C_i}(-\log S_{C_i})\otimes\sL {\buildrel \simeq \over \longrightarrow} 
\sL^{-1}, \textrm{ with } \sL^{-1}=\theta_{C_i}(T_{C_i}(-\log S_{C_i})\otimes E^{2,0}_{C_i}).
$$
The condition (HHP) for all $C_i$ just means that
$\theta_{N_{C_i/Y}}(N_{C_i/Y}\otimes E^{2,0}_{C_i})$ is a direct factor of $\sU^{1,1}$, hence it has degree zero, too.
That implies that $\det F_Y\cdot C_i=0$. By the same argument as in (b) we obtain that 
$(F_Y,\theta_Y)$ corresponds to a sub-VHS over $Y^0$, which uniformizes $Y^0$ as a special subvariety of unitary type.
\end{proof}

\end{document}